\RequirePackage{fix-cm}
\documentclass[smallextended]{svjour3}       
\smartqed  
\usepackage{graphicx}
\usepackage{amssymb,enumerate}
\usepackage{enumerate}
\usepackage{enumitem}
\usepackage{amsmath}

\begin{document}
\title{Powers of conjugacy classes in a finite group
}


\author{A. Beltr\'an \and R.D. Camina \and M.J. Felipe \and C. Melchor 
}

\institute{Antonio Beltr\'an \and Carmen Melchor \at
              Departamento de Matem\'aticas,\\
              Universidad Jaume I, \\
              12071 Castell\'on, Spain\\
              \email{abeltran@uji.es $\cdot$  cmelchor@uji.es}           %
           \and
           Rachel Deborah Camina \at
           Deparment of Pure Mathematics and Mathematical Statistics,\\
            University of Cambridge,\\  Fitzwilliam College, Cambridge, CB3 0DG, UK\\
            \email{rdc26@dpmms.cam.ac.uk}
           \and
           Mar\'{\i}a Jos\'e Felipe \at
           Instituto Universitario de Matem\'atica Pura y Aplicada,\\
            Universidad Polit\'ecnica de Valencia, \\
            46022 Valencia, Spain\\
             \email{mfelipe@mat.upv.es}
}

\date{Received: date / Accepted: date}

\maketitle

\begin{abstract}
Many results have been established that show how the number of conjugacy classes appearing in the product of classes affect the structure of a finite group. The aim of this paper is to show several results about solvability concerning the case in which the power of a conjugacy class is a union of one or two conjugacy classes. Moreover, we show that the above conditions can be determined through the character table of the group.
\end{abstract}

\section{Introduction}\label{intro}
Let $G$ be a finite group. The product of conjugacy classes is a $G$-invariant set, and consequently, is a union of classes. There exist many results about the structure of a finite group regarding the number of conjugacy classes in the product of its classes, some of which are related to the normal structure and the non-simplicity of the group. In this paper we study three problems about the power of a conjugacy class, each of them corresponds to a section.\\

In \cite{products}, Z. Arad and M. Herzog conjectured that in a non-abelian simple group the product of two non-trivial conjugacy classes is not a conjugacy class. The conjecture has received much attention and has been confirmed for several families of simple groups. We propose the following.  

\begin{conjecture}\label{Con1}
In a non-abelian finite simple group the product of $n$ non-trivial conjugacy classes with $n\in \mathbb{N}$ and $n\geq 2$ is not a conjugacy class.
\end{conjecture}

To tackle this conjecture we prove a characterisation of the property using irreducible characters, see Theorem \ref{Char1}. This enables us to prove that Conjecture \ref{Con1} holds for sporadic simple groups.\\

In \cite{GuralnickNavarro}, R.M. Guralnick and G. Navarro confirmed the conjecture of Arad and Herzog for the particular case of a square of a conjugacy class. We prove the following theorem which confirms Conjecture \ref{Con1} for the case when a product of a single non-trivial conjugacy class is considered. We use the notation $x^G$ to denote the conjugacy class of an element $x\in G$.\\

\textbf{Theorem A.} \textit{Let $K=x^G$ be a conjugacy class of a group $G$. There exists $n\in \mathbb{N}$ and $n\geq 2$ satisfying that $K^n$ is a conjugacy class if and only if $$\chi(x)^n=\chi(1)^{n-1}\chi(x^n)$$ for all $\chi \in {\rm Irr}(G)$. In this case, $\langle K\rangle$ is solvable.}\\

To prove the solvability of $\langle K\rangle$ in Theorem A we utilise the Classification of Finite Simple Groups (CFSG). However, we note that in many cases CFSG is not needed, in particular, when the order of the elements in the conjugacy class is prime, or a power of $2$, or if the classes are real. These results are collected in Theorems \ref{NOCFSG} and \ref{NOCFSGR} of Section 2.\\

In Sections 3 and 4 we will focus on two cases when the power of a conjugacy class is a union of exactly two conjugacy classes. In the first case we suppose one of these conjugacy classes is the trivial class, we demonstrate the following theorem. \\

\textbf{Theorem B.} \textit{Suppose that $K$ is a conjugacy class of a group $G$ such that $K^{n}=\{1\} \cup D$ for some $n\in \mathbb{N}$ with $n\geqslant 2$ and $D$ is a non-trivial conjugacy class. Then $KK^{-1}=\{1\} \cup D$ and $\langle K\rangle$ is solvable.}\\

In \cite{Nuestro5}, Theorem B is proved for the particular case $n=2$ without using the CFSG and the structure of $\langle K\rangle$ and $\langle D\rangle$ is determined.\\

In the second case we suppose the two conjugacy classes are inverse to each other. We believe the following to hold. 

\begin{conjecture}\label{Con3} Let $G$ be a group and let $K$ be a conjugacy class. If $K^n=D \cup D^{-1}$ for some $n\in \mathbb{N}$ and $n\geq 2$ and $D$ a conjugacy class, then $\langle K \rangle$ is solvable. In particular, $G$ is not simple.  
\end{conjecture}

We provide the following evidence to support this conjecture.\\

\textbf{Theorem C.} \textit{Let $G$ be a group and let $K$ be a conjugacy class. If $K^n=D \cup D^{-1}$ for some $n\in \mathbb{N}$ and $n\geq 2$ and $D$ a conjugacy class, then either $|D|=|K|/2$ or $|K|=|D|$. In the first case, $\langle K\rangle$ is solvable.}\\

\textbf{Theorem D.} \textit{Let $G$ be a group and let $K=x^G$ be a conjugacy class of $G$. If $K^2=K \cup K^{-1}$, then $\langle K\rangle$ is solvable. Moreover, $x$ is a $p$-element for some prime $p$.}\\

We will also obtain characterizations with irreducible characters of the properties stated in Theorems B and C. These are collected in Theorems \ref{Char2} and \ref{Char3}. All groups are supposed to be finite. 

\section{Powers of classes which are classes}
In this section we prove that Conjecture \ref{Con1} is true for the particular case of the \emph{n}th power of a conjugacy class, $n\geq2$. Furthermore, we obtain an equivalent property in terms of irreducible characters and prove the solvability of the subgroup generated by such a conjugacy class by means of the CFSG.\\

We use the following lemma to prove Theorem \ref{1}, which will be useful to obtain the solvability part of Theorem A. We denote by $\mathbb{C}[G]$ the complex group algebra over the complex field $\mathbb{C}$. Let $K$ be a conjugacy class of $G$ and denote by $\widehat{K}$ the class sum of the elements of $K$ in $\mathbb{C}[G]$.

\begin{lemma}\label{G1}{{\rm (Lemma 2.1 of \cite{GuralnickNavarro})}} Let $x\in G$, where $G$ is a finite group, and let $K=x^G$. Then the following are equivalent:
\begin{enumerate}
\item $\widehat{K}x \in {\rm \bf Z}(\mathbb{C}[G]).$
\item $\widehat{K}x^{-1} \in {\rm \bf Z}(\mathbb{C}[G]).$
\item For each character $\chi \in {\rm Irr}(G)$, either $\chi(x)=0$ or $|\chi(x)|=\chi(1)$.
\end{enumerate}
\end{lemma}

In the next theorem we find a normal subgroup of a group when there is a conjugacy class such that some of its powers is again a conjugacy class, and an equivalent property in terms of the irreducible characters of the group. This result extends the first half of Theorem A of \cite{GuralnickNavarro} in which the authors prove the case $n=2$. The techniques of the proof are the same.

\begin{theorem} \label{1} Let $G$ be a group and $K=x^G$ with $x\in G$, $n\in \mathbb{N}$ and $n\geq 2$. The following assertions are equivalent:
\begin{enumerate}[label=(\alph*)]
\item $K^n$ is a conjugacy class
\item ${\rm \bf C}_G(x)={\rm \bf C}_G(x^n)$ and $N=x^{-1}K=K^{-1}K=[x,G]\unlhd G$
\item ${\rm \bf C}_G(x)={\rm \bf C}_G(x^n)$ and $\chi(x)=0$ or $|\chi(x)|=\chi(1)$ for all $\chi\in$ {\rm Irr}$(G)$.
\end{enumerate}
\end{theorem}

\begin{proof} Let us prove that (a) implies (b). Since $x^n\in K^n$ and $K^n$ is a conjugacy class, it follows that $(x^n)^G=K^n$. Furthermore, for all $2 \leq j \leq n$, we see that $xK^{j-1}\subseteq K^j$ and so $|K|\leq |K^j|\leq |K^n|$. On the other hand, since ${\rm \bf C}_G(x)\subseteq{\rm \bf C}_G(x^n)$, we have $|K^n|\leq |K|$. Thus, $|K|=|K^j|=|K^n|$ and ${\rm \bf C}_G(x)={\rm \bf C}_G(x^n)$. In particular, $xK^{n-1}=K^n$ and $xK=K^2$. Let $y\in K$ then $yK=K^2=xK$ and so $x^{-1}yK=K$. As $y=x^g$ for some $g\in G$ it follows that $[x,g]K=K$. So, for $N=[x, G]=\langle [x,g]\, \, |\, \, g \in G\rangle$ we have $NK=K$ and so $Nx\subseteq K$ and $|N|=|Nx|\leq |K|$. However, as $K=x\{[x,g]\, \, | \, \, g\in G\}\subseteq xN$, it follows that $|K|\leq |xN|=|N|$. Consequently, $K=xN$. Furthermore, $K^{-1}=x^{-1}[x, G]$ and $KK^{-1}=[x, G]$ as required.\\

Suppose (b) and let us see (c). Since $Kx^{-1}=N$, then $\widehat{K}x^{-1}=\widehat{N}$. Also, $\widehat{N}\in {\rm \bf Z}(\mathbb{C}[G])$ since $N\triangleleft G$. Therefore, assertion (c) holds by Lemma 2.1 (3).\\

Assuming (c) now, Lemma \ref{G1} guarantees that $\widehat{K}x$ is central in $\mathbb{C}[G]$, and thus the set $Kx$ is closed under conjugation. Let us see that $K^2=xK$. Clearly, $Kx\subseteq K^2$ and let $x^gx^h\in K^2$ for some $g, h \in G$. Thus, $((x^g)^{h^{-1}}x)^h \in (Kx)^h=Kx$. Therefore, $K^2=xK$. We obtain by induction that $K^n=x^{n-1}K$. Since $|K^n|=|K|=|x^G|=|(x^n)^G|$, then $K^n=(x^n)^G$ and (a) is proved.
\end{proof}

\begin{remark}\label{RemCom} As a consequence of Theorem \ref{1} we have that if $[x, G]=\{{[x,g]\, | \, g \in G}\}$, then $K^n$ is a conjugacy class when $(n, o(x))=1$. 
\end{remark}

It follows, from Theorem \ref{1}, that if $G$ is a finite group with a non-central conjugacy class $K$ such that $K^n$ is a conjugacy class for some $n\geq 2$, then $G$ is not simple. The following corollaries will be useful to prove some results later. 

\begin{corollary}\label{C1} Let $G$ be a group and $K=x^G$ with $x \in G$ such that $K^n$ is a conjugacy class for some $n\in \mathbb{N}$ with $n\geq 2$. Then $|K^r| = |K|$ for all $r\in \mathbb{N}$ and $K^{o(x)+1}=K$ and $K^{o(x)-1}=K^{-1}$. Moreover, $K^m$ is a conjugacy class for all $m\in \mathbb{N}$ such that $(m, o(x))=1$.
\end{corollary}

\begin{proof} Since $K^n$ is a conjugacy class, we know by Theorem \ref{1}(b) that $K=xN$ with $N=KK^{-1}=[x,G]\unlhd G$. Thus, $K^r=x^rN$ and $|K^r|=|N|=|K|$ for all $r \in \mathbb{N}$. Furthermore, if $s=o(x)$, then $K=x^sK\subseteq K^{s+1}$, so $K^{s+1}=K$. Analogously, since $x^{-1}=x^{s-1}\in K^{s-1}$ and $|K^{-1}|=|K|=|K^{s-1}|$, we conclude that $K^{s-1}=K^{-1}$. Finally, let $m \in \mathbb{N}$ such that $(m, o(x))=1$, then ${\rm \bf C}_G(x)={\rm \bf C}_G(x^m)$ and $K^m$ is a class by applying Theorem \ref{1}(b).  
\end{proof}

\begin{corollary}\label{C2} Let $G$ be a group and $K=x^G$ with $x\in G$ such that $K^n=K$ for some $n\in \mathbb{N}$ with $n\geq 2$, then:

\begin{enumerate}[label=(\alph*)]
\item $K^{k(n-1)+r}=K^r$ for every $r, k\in \mathbb{N}$. 
\item $K^{n-1}=[x, G]\unlhd G$. 
\item $\pi(o(x))\subseteq \pi(n-1)$ where $\pi(t)$ denotes the set of primes dividing the number $t$.
\end{enumerate}
\end{corollary}

\begin{proof}
(a) First, let us see that $K^{k(n-1)+1}=K$ for every $k\in \mathbb{N}$. It follows by induction on $k$. It is given if $k=1$. Let us suppose that $K^{k(n-1)+1}=K$ for some $k\in \mathbb{N}$. Then $K^{(k+1)(n-1)+1}=K^{k(n-1)+n}=K^{k(n-1)}K^n=K^{k(n-1)}K=K^{k(n-1)+1}=K$. In general, for every $k, r \in \mathbb{N}$, we have $$K^{k(n-1)+r}=K^{k(n-1)+1+r-1}=KK^{r-1}=K^r.$$

(b) Since $xK^{n-1}\subseteq K^n=K$, we have $|xK^{n-1}|\leq |K|$. We also know that $|K|\leq |xK^{n-1}|$, so $xK^{n-1}=K$. On the other hand, by applying Theorem \ref{1}(b), we obtain $K=x[x, G]$, so $K^{n-1}=[x, G]$.\\

(c) By (a), we know that $K^{k(n-1)+1}=K$ for every $k\in \mathbb{N}$. As a consequence, $o(x)=o(x^{k(n-1)+1})$ for every $k\in \mathbb{N}$. Let $p$ be a prime such that $(p, n-1)=1$. We can find $k$ with $1\leq k<p$ such that $n-1\equiv k\, \, ($mod $p)$. Since $\mathbb{Z}_p$ is a field, there exists $t\in \mathbb{Z}^+$ such that $tk \equiv -1\, \, ($mod $p)$. In fact, $t$ can be taken such that $1\leq t< p$. Now $$t(n-1)\equiv tk\equiv -1\, \,({\rm mod}\, \,p),$$ that is, $t(n-1)+1\equiv 0$ (mod $p)$. Since $o(x)=o(x^{t(n-1)+1})$, we have $(t(n-1)+1, o(x))=1$, and this implies that $p$ does not divide $o(x)$.
\end{proof}

\begin{remark} With the notation of Corollary \ref{C2}, observe that $K^2=K$ cannot happen. Otherwise, by Theorem \ref{1}(b), we have $K=xN$, so $x^2N=xN$, and hence $x\in N$, that is, $K=N$, a contradiction. 
\end{remark}

Let us see an example in which $K^3=D$ with $D\neq K$.

\begin{example} Let $G=\langle a\rangle \rtimes \langle b\rangle$ with $\langle a\rangle  \cong \mathbb{Z}_3$ and $\langle b\rangle \cong \mathbb{Z}_4$. Let $K=b^G$, then $K^3=D\neq K$ and $|K|=3$.
\end{example}

\begin{corollary}\label{C3} Let $G$ be a group and let $\pi$ be a set of primes.  Suppose that for each conjugacy class $K$ of $\pi$-elements of $G$ there exists $n\in \mathbb{N}$ with $n\geq 2$ such that $K^n$ is a conjugacy class . Then $G/${\rm \textbf{O}}$_{\pi'}(G)$ is nilpotent. In particular, if $\pi=\pi(G)$, then $G$ is nilpotent.
\end{corollary}

\begin{proof}
Analogous to the proof of Corollary E of \cite{Nuestro5}.
\end{proof}

\begin{remark} Following Remark \ref{RemCom} we have that if $[x, G]=\{{[x,g]\, | \, g \in G}\}$ for each $\pi$-element $x$ of $G$, by Corollary \ref{C3},  $G/${\rm \textbf{O}}$_{\pi'}(G)$ is nilpotent. In particular, if $\pi=\pi(G)$, then $G$ is nilpotent.
\end{remark}

The following result, which does not require the CFSG, will be useful for our purposes. 

\begin{theorem}\label{G2b}{{\rm (Theorem 3.2(c) of \cite{GuralnickNavarro})}}
Let $G$ be a finite group and let $N$ be a normal subgroup of $G$. Let $x\in G$ be such that all elements of $xN$ are conjugate in $G$. If $x$ is a $p$-element for a prime $p$, then $N$ has a normal $p$-complement.
\end{theorem}

Now we see some particular cases in which the solvability of $\langle K\rangle$ where $K$ is a conjugacy class such that $K^n$ is a class for some $n \in \mathbb{N}$ and $n\geq 2$ can be obtained without using the CFSG. First, we add conditions about the order of the elements of $K$ and later we study the particular case when $K^n$ is a real class.

\begin{theorem}\label{NOCFSG} Let $K$ be a conjugacy class of an element $x$ of a group $G$. Suppose that there exists $n\in \mathbb{N}$ with $n\geq 2$ satisfying that $K^n$ is a conjugacy class. Then:
\begin{enumerate}
\item If $o(x)$ is a prime, then $\langle K\rangle$ is solvable.
\item If $x$ is a $2$-element, then $\langle K\rangle$ is solvable. 
\end{enumerate}
\end{theorem}

\begin{proof}
By Theorem \ref{1}(b) we have $K=xN$ with $N=KK^{-1}=[x, G]\unlhd G$. Let us prove (1). Suppose $x$ is of prime order $p$, we show that ${\rm \bf C}_N(x)$ is a $p$-group. Since $N=x^{-1}K$, if we take some element $x^{-1}x^g\in {\rm \bf C}_N(x)$, then $x\in {\rm \bf C}_G(x^g)$. Thus, $o(x^{-1}x^g)$ divides the least common multiple of $o(x^{-1})$ and $o(x^g)$, so all non-trivial elements of ${\rm \bf C}_N(x)$ have order $p$. In particular, ${\rm \bf C}_N(x)$ is a $p$-group, as wanted. Now, all elements of $xN$ are $G$-conjugate, so $N$ has a normal $p$-complement by Theorem \ref{G2b}. We write $N=P_0L$ with $P_0$ a $p$-group and $L\unlhd N$ a $p'$-group. Since ${\rm \bf C}_L(x)\subseteq {\rm \bf C}_N(x)$ and ${\rm \bf C}_N(x)$ is a $p$-group, we conclude that ${\rm \bf C}_L(x)=1$ and since $o(x)=p$, we deduce that $L$ is nilpotent by Thompson's Lemma (see for instance Theorem 2.1 in Chapter 10 of \cite{Goreinstein}). As a result, $N$ is solvable and $\langle K\rangle=\langle x\rangle N$ is solvable too.\\

Now, we prove (2). By Theorem \ref{G2b}, $N$ has a normal $2$-complement, and consequently, as $\langle K\rangle/N$ is a 2-group, then $\langle K\rangle$ has a normal 2-complement as well. By Feit-Thompson's Theorem, we conclude that $\langle K\rangle$ is solvable.
\end{proof}

\begin{theorem}\label{NOCFSGR} Let $K$ be a conjugacy class of an element $x$ of a group $G$. Suppose that there exists $n\in \mathbb{N}$ with $n\geq 2$ satisfying that $K^n=D$ where $D$ is a real conjugacy class, then $\langle K\rangle$ is solvable. Also, $D^3=D$ and $D$ is a class of a $2$-element. In particular, 
\begin{enumerate}
\item [(a)] If $n=2^a$ for some $a\in \mathbb{N}$, then $|K|$ is odd and $o(x)=2^{a+1}$.
\item [(b)] If $D=K$, then $x$ is a $2$-element and $K^m=K$ for every odd number $m$. Also, $K^2=[x,G]\unlhd G$.
\end{enumerate}
\end{theorem}

\begin{proof}
We have $$K^n=D=D^{-1}=(K^n)^{-1}=(K^{-1})^{n}.$$ By Corollary \ref{C1}, we get $|K|=|K^n|=|D|$ and we conclude that $x^{n-1}K=K^n$. Hence, $$K=(x^{n-1})^{-1}K^n=(x^{n-1})^{-1}(K^{-1})^n=(x^{-1})^{n-1}(K^{-1})^{n}\subseteq (K^{-1})^{2n-1}.$$ By applying Corollary \ref{C1} to $K^{-1}$, we obtain $K=(K^{-1})^{2n-1}$. Thus, $K^{-1}=K^{2n-1}$. We have $K\subseteq KKK^{-1}=K^2K^{2n-1}$ and as $|K| = |K^{2n+1}|$, by Corollary \ref{C1}, it follows that $K^{2n+1}=K$. Thus, by multiplying both sides by $K^{n-1}$, we obtain $K^{3n}=K^n$, so $D^3=D$ and $D$ is a conjugacy class of a $2$-element by Corollary \ref{C2}(c). By Theorem \ref{NOCFSG} (2) we have that $\langle D\rangle$ is solvable. Since $K^{2n+1}=K$, we know by Corollary \ref{C2}(b) that $K^{2n}=D^2=N=[x, G]$, so $N$ is solvable. By Theorem \ref{1}(b) we have $K=xN$ with $N=[x, G]\unlhd G$, so $\langle K\rangle$ is solvable.\\

Let us prove the particular case (a). Suppose that $|K|=|D|$ is odd. Then $D$ is a real class of odd size and then $o(x^n)=2$. This means that $o(x)=2^{a+1}$ and the result is proved. Assume that $|K|=|D|$ is even and we are going to get a contradiction. By Theorem \ref{1}, $N=KK^{-1}=x^{-1}K \unlhd G$ and we write $$x^{-1}K=\{1\} \cup D_1\cup \cdots \cup D_m$$ where each $D_i$ is the conjugacy class of an element $x^{-1}x^g\neq 1$ for some $g \in G$. Since $|x^{-1}K|=|K|$ is even, there exists in $x^{-1}K$, at least one conjugacy class of odd size. Let $\Omega=\{D_i\, | \, |D_i|\, \, {\rm odd}\}$ and we necessarily have that $|\Omega|$ is odd. Thus, there exists a real conjugacy class $D_k$ with $k \in \{1, \cdots, m\}$ such that $|D_k|$ is odd, that is, $D_k$ is a non-trivial conjugacy class of elements of order 2. We write $D_k=t^G$. Since $|D_k|$ is odd, then ${\rm \bf C}_G(t)$ contains a Sylow $2$-subgroup of $G$. Observe that $x^n$ is a 2-element by the previous part. Consequently, $x$ is a 2-element. By taking conjugates, we can suppose that $\langle x\rangle \subseteq P \subseteq {\rm \bf C}_G(t^s)$ for some $P \in {\rm Syl}_2(G)$ and some $s \in G$. We write $t^s=x^{-1}x^g$ for some $g\in G$. Observe that $x^g \neq x$ because $t\neq 1$. We have $x \in {\rm \bf C}_G(x^{-1}x^g)$, so $x$ and $x^g$ commute. Since $o(x^{-1}x^g)=2$, then $(x^2)^g=x^2$, so $g \in {\rm \bf C}_G(x^2)$. On the other hand, as $|K^n|=|K|$, we have ${\rm \bf C}_G(x^n)={\rm \bf C}_G(x)$, and since ${\rm \bf C}_G(x^2)\subseteq {\rm \bf C}_G(x^n)$, this leads to $t=1$, a contradiction.\\

Finally, (b) directly follows from the first part of this theorem. 
\end{proof}

\begin{remark}
Observe that if a real conjugacy class $K$ satisfies that there exists $n\in \mathbb{N}$ such that $K^n=D$ where $D$ is a conjugacy class, then $D$ is also a real class. However, if a class $K$ satisfies that there exists $n\in \mathbb{N}$ such that $K^n=D$ and $D$ is real, then $K$ need not be real. A trivial example of this situation occurs in $\mathbb{Z}_4$. We have studied this case in the previous theorem. 
\end{remark}

We use the following result appearing in \cite{GuralnickNavarro} to obtain the solvability of the subgroup generated by a conjugacy class satisfying the conditions of Theorem \ref{1}, and consequently, to prove the solvability part of Theorem A. The CFSG is required. 

\begin{theorem}\label{G2a}{{\rm (Theorem 3.2(a) of \cite{GuralnickNavarro})}}
Let $G$ be a finite group and let $N$ be a normal subgroup of $G$. Let $x\in G$ be such that all elements of $xN$ are conjugate in $G$. Then $N$ is solvable. 
\end{theorem}

In the next result the solvability in Theorem A is obtained by means of the CFSG. In fact, Theorems \ref{1} and \ref{S1} are extensions of some parts of Theorem A of \cite{GuralnickNavarro}, in which the authors prove similar results for the square of a conjugacy class.

\begin{theorem}\label{S1} Let $K$ be a conjugacy class of a group $G$ such that there exists $n\in \mathbb{N}$ satisfying that $K^n$ is a conjugacy class. Then $\langle K\rangle$ is solvable.
\end{theorem}

\begin{proof} By Theorem \ref{1}(b), we have $K=xN$ with $N=[x, G]$. Thus, $N$ is solvable by applying Theorem \ref{G2a}. As a consequence, $\langle K\rangle=\langle x\rangle N$ is solvable.
\end{proof}

Next, we obtain a characterization in terms of characters of the fact that the product of $n$ conjugacy classes, for some $n\in \mathbb{N}$, is again a conjugacy class. This extends the case in which the product of two classes is a class (see for instance \cite{MooriViet}) and this is useful to prove Conjecture \ref{Con1} for the sporadic simple groups for some values of $n$. In particular, we obtain such a characterization for the case in which the power of a class is a class and so, the first part of Theorem A. We refer the reader to Chapter 3 of \cite{Isaacs} for a detailed presentation of character and class sums properties. 

\begin{theorem}\label{Char1} Let $K_1,\cdots, K_n $ be conjugacy classes of a group $G$ and write $K_i={x_i}^G$ with $x_i \in G$. Then $K_1\cdots K_n=D$ where $D=d^G$ if and only if $$\chi(x_1)\cdots \chi(x_n)=\chi(1)^{n-1}\chi(d)$$ for all $\chi \in {\rm Irr}(G)$. In particular, if $K$ is a conjugacy class of $G$ and $x\in K$, then $K^n$ is a conjugacy class for some $n\in \mathbb{N}$ if and only if 
\begin{equation}
\chi(x)^n=\chi(1)^{n-1}\chi(x^n)
\end{equation} for all $\chi \in {\rm Irr}(G)$.
\end{theorem}

\begin{proof} Let $\chi \in {\rm Irr}(G)$ and let $R$ be an irreducible representation associated to $\chi$. We know that $R$ can be linearly extended to $\mathbb{C}[G]$ and $R(\widehat{K})\in {\rm \bf Z}(\mathbb{C}[G])$ commutes with $R(g)$ for all $g\in G$. We denote by $\widehat{K_i}$ the sum of all elements in $K_i$ in the group algebra $\mathbb{C}[G]$. We know that $$R(\widehat{K_i})=w_{\chi}(\widehat{K_i})I$$ where $$w_{\chi}(\widehat{K_i})=\frac{|K_i|\chi(x_i)}{\chi(1)}$$ and $I$ is the identity matrix.\\

Assume that $K_1\cdots K_n=D$. We write $\widehat{K_1}\cdots \widehat{K_n} =m\widehat{D}$ with $m\in \mathbb{N}$. Thus, the hypothesis implies that \begin{equation*}
R(\widehat{K_1}\cdots \widehat{K_n})=R(\widehat{K_1})\cdots R(\widehat{K_n})=mR(\widehat{D})
\end{equation*}
and
\begin{equation*}
w_{\chi}(\widehat{K_1})\cdots w_{\chi}(\widehat{K_n}) =mw_{\chi}(\widehat{D}).
\end{equation*} 
Consequently, $$\frac{|K_1|\cdots |K_n|\chi(x_1)\cdots \chi(x_n)}{\chi(1)^n}=m\frac{|D|\chi(d)}{\chi(1)}$$ and since $|K_1|\cdots |K_n|=m|D|$, we have $$\chi(x_1)\cdots \chi(x_n)=\chi(1)^{n-1}\chi(d)$$for every $\chi \in {\rm Irr}(G)$.\\

Let us prove the converse. Assume that the equation with characters holds. We know that $K_1\cdots K_n=D_1 \cup \cdots \cup D_r$ with $D_i$ a conjugacy class for all $1\leq i\leq r$. We write $\widehat{K_1}\cdots \widehat{K_n}=m_1\widehat{D_1}+ \cdots +m_r\widehat{D_r}$, where $m_i$ is the multiplicity of $\widehat{D_i}$ in the product. We have $|K_1|\cdots |K_n|=m_1|D_1|+\cdots +m_r|D_r|$. As above,
\begin{equation*}
R(\widehat{K_1}\cdots \widehat{K_n})=R(\widehat{K_1})\cdots R(\widehat{K_n})=m_1R(\widehat{D_1})+\cdots +m_rR(\widehat{D_r})
\end{equation*}
and by hypothesis we know
$$\chi(x_1)\cdots \chi(x_n)=\chi(1)^{n-1}\chi(d)$$ for all $\chi \in {\rm Irr}(G)$.
Thus, 
\begin{equation*}
|K_1|\cdots |K_n|\chi(d)=m_1|D_1|\chi(d_1)+\cdots +m_r|D_r|\chi(d_r)
\end{equation*} 
where $D_i=d_i^G$ and\begin{equation*}
|K_1|\cdots |K_n|\chi(d)\overline{\chi(d)}=m_1|D_1|\chi(d_1)\overline{\chi(d)}+\cdots +m_r|D_r|\chi(d_r)\overline{\chi(d)}.
\end{equation*} 
From this equation we obtain 

$$|K_1|\cdots |K_n|\sum_{\chi \in {\rm Irr}(G)}\chi(d)\overline{\chi(d)}=|K_1|\cdots |K_n||{\rm \bf C}_G(d)|=$$
$$=m_1|D_1|\sum_{\chi \in {\rm Irr}(G)}\chi(d_1)\overline{\chi(d)}+\cdots +m_r|D_r|\sum_{\chi \in {\rm Irr}(G)}\chi(d_r)\overline{\chi(d)}.$$
Then $D_i=D$ for some $i$. Without loss of generality, suppose that $D_1=D$ and we have $m_i=0$ for all $i\neq 1$. This means that $K_1\cdots K_n=D$. 
\end{proof}

\begin{remark} Recall that the extended covering number of a group is the smallest integer $r$ such that the product of $r$ conjugacy classes is the whole group for all classes. In \cite{Zisser}, it is shown that the extended covering number of the sporadic simple groups is at most 7. By using the character tables of the sporadic groups we have checked that for each of them and for each $n$-tuple of conjugacy classes for $n=3, 4, 5, 6$ there is some irreducible character which does not satisfy equation (1) of Theorem \ref{Char1}. The case $n=2$ obviously corresponds to Arad and Herzog's conjecture, which is already known to be satisfied by the sporadic simple groups. Therefore, Conjecture \ref{Con1} holds for the sporadic simple groups. 
\end{remark}

\noindent {\it Proof of Theorem A.} It is a direct consequence of Theorems \ref{S1} and \ref{Char1}.

\bigskip
\section{Powers which are a union of the trivial class and another class}
In this section we study the case in which the power of a conjugacy class is a union of two conjugacy classes one of them being the trivial class. We first prove
a particular case satisfying the conjecture of Arad and Herzog which will also be useful in further proofs.

\begin{lemma}\label{L1}
Let $G$ be a group and $K$, $L$ and $D$ non-trivial conjugacy classes of $G$ such that $KL=D$ with $|D|=|K|$. Then $G$ possesses a solvable proper normal group which is $\langle LL^{-1}\rangle$. 
In particular, $\langle L\rangle$ is solvable.
\end{lemma}

\begin{proof}
 Let $x\in L$. Then $xK=D=x^{g}K$ for all $g \in G$. Consequently, $K=x^{-1}x^{g}K$ for all $g\in G$. Let $N=\langle x^{-1}x^g \, \mid \, g \in G\rangle=\langle LL^{-1}\rangle$ is normal in $G$ and then $NK=K$. Since $K$ is union of cosets of $N$, then $|N|$ divides $|K|$. Then $N$ is proper in $G$. In addition, since all elements in $xN$ are conjugate, $N$ is solvable by Theorem \ref{G2a}. Furthermore, it is an elementary fact that $\langle L\rangle=\langle x \rangle[x, G]=\langle x \rangle N$, so $\langle L\rangle/N$ is cyclic, and  consequently, $\langle L\rangle$ is solvable.
\end{proof}

We also need the following two results due to Guralnick and Robinson, which appeal to the CFSG, as well as Kazarin's extension of Burnside's Lemma.

\begin{theorem}{\rm (Theorem A of \cite{GurRob})}\label{Teo2}  Let $G$ be a finite group and $p$ a prime. Let $x\in G$ be an element of order $p$ such that $[x, g]$ is a $p$-element for every $g\in G$. Then $x \in$ {\rm \textbf{O}}$_p(G)$.
\end{theorem}

\begin{theorem}{\rm (Theorem 4.1 of \cite{GurRob})}\label{Teo3} Let $G$ be a finite group and $p$ a prime. If $x\in G$ has order $p$ and is not central modulo {\rm \textbf{O}}$_{p'}(G)$, then $x$ commutes with some conjugate $x^g\neq x$.
\end{theorem}

\begin{theorem}{\rm (Kazarin,  Theorem 15.7 of \cite{Huppert})}\label{Teo5}
Suppose $1 \neq g \in G$ and $|g^G|=p^a$, where $p$ is a prime. Then $g^G$ generates a solvable normal subgroup of $G$.
\end{theorem}

We are ready to prove Theorem B.\\

\noindent {\it Proof of Theorem B.} We write $K^{n-1}=A_1 \cup \cdots \cup A_m$ where $A_i$ are distinct conjugacy classes for $i=1, \ldots, m$. So $K^n=KA_1\cup \cdots \cup KA_m= \{1\} \cup D$. Thus, $1 \in KA_i$ for some $i$ and we assume without loss of generality $i=m$. So we write $K^{n-1}=K^{-1}\cup A_1 \cup \cdots \cup A_{m-1}$. Then $KK^{-1}\subseteq K^n=\{1\} \cup D$ and we have either $KK^{-1}=\{1\}$ or $KK^{-1}=\{1\} \cup D$. In the first case $K=\{ x\}$, that is, $x$ is central in $G$, so $K^n=\{x^n\}$ and this is not possible. Therefore, $KK^{-1}=\{1\} \cup D$.\\

To prove the solvability of $\langle K\rangle$ we argue by minimal counterexample, so let $G$ be a minimal counterexample. Write $K=x^G$ with $x\in G$ and we distinguish two possibilities: $x^n= 1$ and $x^n\neq 1$. Assume first that $x^n\neq 1$.
  If $m=1$, where $m$ is as in the above paragraph, then $K^{n-1}=K^{-1}$. In addition, by Corollary \ref{C1} we know that $K^{o(x)-1}=K^{-1}$ and $K^{o(x)+1}=K$. So, since $K^{n-1}=K^{-1}=K^{o(x)-1}$, we deduce that $K^{n+1}=K^{o(x)+1}=K$ and it necessarily follows that $KD=K$. By Lemma \ref{L1}, $\langle D\rangle=\langle KK^{-1}\rangle$ is solvable, so the case $m=1$ is finished. Suppose now that $m>1$, that is,  there exists $i \in \{2, \ldots, m\}$ such that $KA_i=D$. Then  $|K|\leqslant |D|$. On the other hand, since $x^n\neq 1$, then $D=(x^n)^G$ and ${\rm \bf C}_G(x)\subseteq {\rm \bf C}_G(x^n)$ implies that $|D|$ divides $|K|$. As a result $|D|=|K|$. We can apply  Lemma \ref{L1} and we obtain that $\langle A_i\rangle$ is solvable. Now, consider $\overline{G}=G/\langle A_i\rangle$. From the hypothesis, we have $\overline{K}^{n}=\{\overline{1}\} \cup \overline{D}$ where $\overline{K}$ denotes the corresponding class in $\overline{G}$. If $\overline{K}^{n}=\{\overline{1}\}$, then $\overline{K}$ is central and if $\overline{K}^{n}=\overline{D}$, then $\langle \overline{K}\rangle$ is solvable by Theorem $A$. Otherwise, by minimal counterexample we get that $\langle\overline{K}\rangle$ is solvable, so  $\langle K\rangle$ is solvable too, a contradiction.\\

 For the rest of the proof we assume that $x^n=1$. First, we prove that $n$ can be assumed to be prime. Suppose that the theorem holds for a prime, that is,  $K^{p}=\{1\} \cup D$ with $p$  prime implies that $\langle K\rangle$ is solvable.  Suppose that $n$ is not prime and write $n=pt$ for a prime $p$  and $t>1$. Write $$K^{t}=C_1 \cup \cdots \cup C_s$$ where $C_i$ are conjugacy classes of $G$ for all $1 \leq i \leq s$. Since $$K^{n}=K^{pt}=(C_1\cup \cdots \cup C_s)^p=\{1\} \cup D,$$ we have $C_{i}^{p}\subseteq \{1\} \cup D$ for every $i$ and there are three possibilities: $C_i^{p}=\{1\}$, $C_{i}^{p}=D$ or $C_{i}^{p}=\{1\} \cup D$. If $C_i^{p}=\{1\}$, then trivially $\langle C_{i}\rangle \leq {\rm \textbf{Z}}(G)$, so $\langle C_i\rangle$ is solvable. If $C_{i}^{p}=D$, then $\langle C_{i}\rangle$ is solvable by Theorem A. Finally, if $C_{i}^{p}=\{1\} \cup D$, then $\langle C_i\rangle$ is solvable by our assumption. Now, we denote $\overline{G}=G/\langle C_i\rangle$ for some non-trivial class $C_i$. Notice that $\overline{K}^{n}= \{\overline{1}\}\cup \overline{D}$. Arguing similarly to
above leads to the fact that $\langle \overline{K}\rangle$ is solvable. Thus, $\langle K\rangle$ is solvable too, a contradiction. \\

Therefore, for the rest of the proof we assume that $K^{p}=\{1\} \cup D$ with $p$  prime, and hence we are assuming that $o(x)=p$. Let $N$ be a minimal normal subgroup of $G$. Arguing as in the above paragraph, that is, by transferring into the quotient $G/N$, using Theorem A and minimality, it easily follows that $N$ is the only minimal normal subgroup of $G$ and that it is a direct  product of  isomorphic simple groups. We will prove that $N$ is solvable, and this contradiction will complete the proof. Set $D=t^G$ with $t\in G$.  If $t$ is a $p$-element, since $KK^{-1}=\{1\} \cup D$, then $x^{-1}x^g=[x, g]$ is a $p$-element for every $g \in G$. By Theorem \ref{Teo2} we have $x\in {\rm \textbf{O}}_{p}(G)\neq 1$. Consequently, $\langle K\rangle \leq {\rm \textbf{O}}_{p}(G)$, which implies that $\langle K\rangle$ is solvable. Thus, we can assume that $o(t)\neq p$. If $x$ is non-central modulo ${\rm \textbf{O}}_{p'}(G)$, then by Theorem \ref{Teo3}, $x$ commutes with some conjugate $x^g\neq x$ , so in particular $o(x^{-1}x^g)=p$,  a contradiction. Therefore, $x$ can be assumed to be central modulo ${\rm \textbf{O}}_{p'}(G)$. Also, if ${\rm \textbf{O}}_{p'}(G)=1$, then $x\in {\rm \textbf{Z}}(G)$ and there is nothing to prove. We assume then that ${\rm \textbf{O}}_{p'}(G)\neq 1$, and by minimality $N\leq {\rm \textbf{O}}_{p'}(G)$. Moreover, if $x$ centralizes ${\rm \textbf{O}}_{p'}(G)$, then $x\in {\bf C}_G(N)\neq 1$, which forces $N$ to be abelian, and the proof is finished. Therefore, we can assert that there exists a prime $q$ dividing $|{\rm \textbf{O}}_{p'}(G):{\rm \textbf{C}}_{{\rm \textbf{O}}_{p'}(G)}(x)|$ and hence, by coprime action, there exists $Q \in {\rm Syl}_{q}({\rm \textbf{O}}_{p'}(G))$ such that $Q^x=Q$, so $1\neq [x, Q]\subseteq Q$. Let $[x,g]$ be a non-trivial $q$-element of $[x, Q]$. Since $[x,g]=x^{-1}x^g\in K^{-1}K=\{1\} \cup D$, then the elements of $D$ are $q$-elements. In particular the prime $q$ must be unique, that is, $q^a=|{\rm \textbf{O}}_{p'}(G):{\rm \textbf{C}}_{{\rm \textbf{O}}_{p'}(G)}(x)|$ with $a\geq 1$. Since $|N:{\rm \textbf{C}}_{N}(x)|$ divides $|{\rm \textbf{O}}_{p'}(G):{\rm \textbf{C}}_{{\rm \textbf{O}}_{p'}(G)}(x)|$, we have $|N:{\rm \textbf{C}}_{N}(x)|=q^b$ for some $b\leq a$. As a consequence $|N\langle x\rangle: {\rm \textbf{C}}_{N\langle x\rangle}(x)|=q^b$, so we can apply Theorem \ref{Teo5} and $\langle x^{N\langle x\rangle}\rangle$ is solvable. Now, it is elementary that $\langle x^{N\langle x\rangle}\rangle=\langle x\rangle[N\langle x\rangle, \langle x\rangle ]=\langle x\rangle[N, x]$. We conclude that $1\neq [N,x]$ is a normal solvable subgroup of $N\langle x\rangle$. This certainly leads to the solvability of $N$, and this is the final contradiction. $\Box$\\

We have seen that $K^n=\{1\} \cup D$ implies that $KK^{-1}=\{1\}\cup D$ and this property was characterized in Theorem B of \cite{Nuestro6} in terms of characters. Thus, the hypothesis of Theorem B implies the following equality with characters.

\begin{corollary}
Let $G$ be a group and $x, d \in G \setminus \{1\}$. Let $K=x^G$, $D=d^G$ such that $K^n=\{1\} \cup D$ for some $n\in \mathbb{N}$. Then for every $\chi \in$ {\rm Irr}$(G)$ $$|K||\chi(x)|^2=\chi(1)^2+(|K|-1)\chi(1)\chi(d).$$
\end{corollary}

\begin{proof}
By Theorem B we know that $KK^{-1}=\{1\} \cup D$ and the result follows by Theorem B of \cite{Nuestro6}.
\end{proof}

\begin{example} Let us show two examples of the situation $K^n=\{1\}\cup D$ with $n=3$. In the first, $x^n=1$ and in the second $x^n\neq 1$. Let $G=A_4$ and $K=(1\, 2\, 3)^G$, which satisfies $|K|=4$ and $o((1\, 2\, 3))=3$. Furthermore, $K^3=\{1\} \cup D$ where $D=((1\, 2)(3\, 4))^G$. On the other hand, let $G=(\mathbb{Z}_7\rtimes \mathbb{Z}_9)\rtimes \mathbb{Z}_2$ having a conjugacy class $K$ of elements of order 21 satisfying that $K^3=1\cup D$ and $|K|=6$ where $D$ is a class of elements of order 7 and $|D|=6$. In this example, $\langle K\rangle\cong \mathbb{Z}_{21}$.
\end{example}

\begin{remark}
We have seen that $K^n= \{1\} \cup D$ implies that $KK^{-1}=\{1\} \cup D$. However, the converse is not true. Let $G=SL(2,3)$ and let $K$ be one of the two conjugacy classes of elements of order 6 which satisfies $|K|=4$. It follows that $KK^{-1}=\{1\} \cup D$ where $D$ is the unique conjugacy class such that $|D|=6$. However, there is no $n\in \mathbb{N}$ with $K^n=\{1\} \cup D$.  
\end{remark}

In \cite{products}, Arad and Herzog published the following result about the multiplicity of a conjugacy class in the product of conjugacy classes. We will use it for the particular case of the power of a class in Theorems \ref{Char2} and \ref{Char3}.

\begin{theorem}{{\rm (Lemma 10.1 of \cite{products})}}\label{multi}
Let $G$ be a group and let $K_1, \ldots, K_r$ be the conjugacy classes of the elements $x_1, \ldots, x_r$, respectively such that $K_1\cdots K_r=D_1\cup \cdots \cup D_t$ where $D_1, \ldots, D_t$ are the conjugacy classes of the elements $d_1, \ldots, d_t$, respectively. Then $$\prod_{i=1}^r\widehat{K_i}=\sum_{j=1}^{t}\alpha_j\widehat{D_j},$$ where $$\alpha_j=\frac{\prod_{i=1}^r|K_i|}{|G|}\sum_{\chi \in {\rm Irr}(G)}\frac{\left(\prod_{i}^{r}\chi(x_i)\right)\overline{\chi(d_j)}}{\chi(1)^{r-1}}$$ for $j=1, \ldots, t$. In particular, if $K=x^G$ and $K^r=D_1 \cup \cdots \cup D_t$, then $$\widehat{K^r}=\sum_{j=1}^{t}\alpha_j\widehat{D_j},$$ and $$\alpha_j=\frac{|K|^r}{|G|}\sum_{\chi \in {\rm Irr}(G)}\frac{\chi(x)^r\overline{\chi(d_j)}}{\chi(1)^{r-1}}.$$
\end{theorem}

\begin{theorem}\label{Char2} Let $G$ be a finite group and let $K$ be a conjugacy class of an element $x\in G$. Then $K^n=\{1\} \cup D$ where $D=d^G\neq \{1\}$ if and only if there exist positive integers $m_1$ and $m_2$ such that $$\chi(x)^n|K|^n=\chi(1)^{n-1}(m_1\chi(1)+m_2|D|\chi(d))$$ for all $\chi \in {\rm Irr}(G)$ and $|K|^n=m_1+m_2|D|$ where

$$m_1=\frac{|K|^n}{|G|}\sum_{\chi \in {\rm Irr}(G)}\frac{\chi(x)^n}{\chi(1)^{n-2}}$$ and $$m_2=\frac{|K|^n}{|G|}\sum_{\chi \in {\rm Irr}(G)}\frac{\chi(x)^n\overline{\chi(d)}}{\chi(1)^{n-1}}.$$

\end{theorem}

\begin{proof} Assume that $K^n=\{1\} \cup D$ and we write $\widehat{K^n}=m_1\widehat{1}+m_2\widehat{D}$ where $m_1$ and $m_2$ are positive integers that can be determined by the character table by using Theorem \ref{multi}. Then $|K|^n=m_1+m_2|D|$. Let $\chi \in {\rm Irr}(G)$ and let $R$ and $w_\chi$ be as in Theorem \ref{Char1}. We have \begin{equation*}
R(\widehat{K^n})=R(\widehat{K})^n=m_1R(\widehat{1})+m_2R(\widehat{D})
\end{equation*}
and
\begin{equation*}
w_{\chi}(\widehat{K})^n=m_1w_{\chi}(\widehat{1})+m_2w_{\chi}(\widehat{D}).
\end{equation*} 
Then $$|K|^{n}\frac{\chi(x)^n}{\chi(1)^n}=m_1+m_2\frac{|D|\chi(d)}{\chi(1)}$$ and so $$\chi(x)^n|K|^n=\chi(1)^{n-1}(m_1\chi(1)+m_2|D|\chi(d))$$ for all $\chi \in {\rm Irr}(G)$, as wanted. \\

Conversely, assume that there exist $m_1$ and $m_2$ satisfying the equalities with characters. We write $K^n=D_1 \cup \cdots \cup D_r$ with $D_i$ a conjugacy class for all $1\leq i\leq r$. We write $\widehat{K^n}=n_1\widehat{D_1}+ \cdots +n_r\widehat{D_r}$ and notice that $|K|^n=n_1|D_1|+\cdots +n_r|D_r|$. Let $\chi \in {\rm Irr}(G)$ and let $R$ and $w_\chi$ be as before. Then \begin{equation*}
R(\widehat{K^n})=R(\widehat{K})^n=n_1R(\widehat{D_1})+\cdots +n_rR(\widehat{D_r})
\end{equation*}
and by hypothesis we know
$$\chi(x)^n|K|^n=\chi(1)^{n-1}(m_1\chi(1)+m_2|D|\chi(d))$$ for all $\chi \in {\rm Irr}(G)$.
Therefore, 
\begin{equation}\label{CAR11}
m_1\chi(1)+m_2|D|\chi(d)=n_1|D_1|\chi(d_1)+\cdots +n_r|D_r|\chi(d_r).
\end{equation} 
By multiplying both sides by $\overline{\chi(d)}$ we get
 \begin{equation*}
m_1\chi(1)\overline{\chi(d)}+|D|m_2\chi(d)\overline{\chi(d)}=n_1|D_1|\chi(d_1)\overline{\chi(d)}+\cdots +n_r|D_r|\chi(d_r)\overline{\chi(d)}.
\end{equation*} 
From this equation, we obtain 

$$m_1\sum_{\chi \in {\rm Irr}(G)}\chi(1)\overline{\chi(d)}+|D|m_2\sum_{\chi \in {\rm Irr}(G)}\chi(d)\overline{\chi(d)}=|D|m_2|{\rm \bf C}_G(d)|=$$
$$=n_1|D_1|\sum_{\chi \in {\rm Irr}(G)}\chi(d_1)\overline{\chi(d)}+\cdots +n_r|D_r|\sum_{\chi \in {\rm Irr}(G)}\chi(d_r)\overline{\chi(d)}.$$
Then $D_i= D=d^G$ for some $i$. Without loss of generality, we can assume that $D_1=D$. Now, if we multiply both sides of Eq.(\ref{CAR11}) by $\chi(1)$ and argue similarly we conclude that $D_2=\{1\}$ and $n_i=0$ for all $i\neq 1,2$. This means that $K^n=\{1\} \cup D$.
\end{proof}

\section{Powers which are a union of a class and its inverse}
In this section we are going to study the case in which the power of a conjugacy class is a union of two classes, one of them being the inverse of the other, and we prove Theorems C and D. We use the CFSG in all results except in Theorem \ref{Char3}.

\begin{remark} If $K=x^G$ with $x\in G$ and $K^n=D \cup D^{-1}$ for some $n \in \mathbb{N}$ with $D \neq D^{-1}$, then $K$ is non-real. Suppose that $K$ is real and $x^n \in D$. We have that $x^{-1}=x^g$ for some $g \in G$. Then $(x^n)^g=(x^g)^n=(x^{-1})^n=x^{-n} \in D\cap D^{-1}$, a contradiction. 
\end{remark}

We give the proof of Theorem C, which demonstrates that Conjecture \ref{Con3} is true when $|D|=|K|/2$.\\

\noindent {\it Proof of Theorem C.} Notice that if $D=D^{-1}$, we have the hypothesis of Theorem A and the result immediately follows. So we assume that $D$ is not a real class. Let $K=x^G$. We know that either $D=(x^n)^G$ or $D^{-1}=(x^n)^G$. Without loss of generality, we may assume that $D=(x^n)^G$. Since ${\rm {\bf C}}_G(x)\subseteq {\rm {\bf C}}_G(x^n)$ we have that $|D|$ divides $|K|$. Furthermore, it follows that $|K|\leq |K^n|=2|D|$, that is, $|K|/2 \leq |D|$. Consequently, either $|D|=|K|/2$ or $|K|=|D|$. Suppose that $|D|=|K|/2$. Since $|K^n|=2|D|=|K|$, we deduce that $|K^i|=|K|$ for all $i\leq n$. Thus, $xK=K^2$ and similarly, if $y\in K$, we get $yK=K^2$. By arguing as in Theorem \ref{1} it can be obtained that $K=xN$ where $N=[x, G]\unlhd G$. By Theorem \ref{G1}, $N$ is solvable and consequently, $\langle K\rangle=\langle x\rangle N$ is also solvable. $\Box$

\begin{example} We are going to see that both cases of Theorem C are possible. Let $G=\mathbb{Z}_8 \rtimes \mathbb{Z}_2=M_{16}=\langle a, x\, \, | \, \, a^8=x^2=1, \, \, a^{x^{-1}}=a^5\rangle$ and $K=a^G$. It follows that $K^2= D \cup D^{-1}$, $|K|=2$ and $|D|=1$. On the other hand, let $G=\mathbb{Z}_2 \times (\mathbb{Z}_7 \rtimes \mathbb{Z}_3)$ and $K=x^G$ where $o(x)=14$. We have $K^2= D\cup D^{-1}$ and $|K|=|D|=3$.
\end{example}

In Theorem D we prove Conjecture \ref{Con3} in the particular case $n=2$ and $D=K$. We will work in the complex group algebra $\mathbb{C}[G]$ and we will use the following properties. Let $g_1, \ldots, g_k$ be representatives of the conjugacy classes of a finite group $G$. Let $\widehat{S}=\sum_{i=1}^{k}n_i\widehat{g_i^G}$ with $n_i\in \mathbb{N}$ for $1 \leq i \leq k$. We write $(\widehat{S}, \widehat{g_i^G})=n_i$ following \cite{AradFisman}.

\begin{lemma}\label{Le1} If $D_{1}$, $D_{2}$ and $D_{3}$ are conjugacy classes of a finite group $G$, then
\begin{enumerate}
\item $(\widehat{D_{1}}\widehat{D_{2}}, \widehat{D_{3}})=(\widehat{D_{1}^{-1}}\widehat{D_{2}^{-1}}, \widehat{D_{3}^{-1}})$
\item $(\widehat{D_{1}}\widehat{D_{2}}, \widehat{D_{3}})=|D_{2}||D_{3}|^{-1}(\widehat{D_{1}}\widehat{D_{3}^{-1}}, \widehat{D_{2}^{-1}})$
\item $(\widehat{D_{1}}\widehat{D_{2}}, \widehat{D_{1}})=|D_{2}||D_{1}|^{-1}(\widehat{D_{1}}\widehat{D_{1}^{-1}}, \widehat{D_{2}^{-1}})=(\widehat{D_{2}}\widehat{D_{1}^{-1}}, \widehat{D_{1}^{-1}})=(\widehat{D_{2}^{-1}}\widehat{D_{1}}, \widehat{D_{1}})$.
\end{enumerate}
\end{lemma}

\begin{proof}
See the proof of Theorem A of \cite{AradFisman}.
\end{proof}

\noindent {\it Proof of Theorem D.} We argue by induction on $|G|$. We write $\widehat{K}^2=\alpha\widehat{K}+\beta\widehat{K^{-1}}$ with $\alpha, \beta \in \mathbb{Z}^{+}$ and $\alpha=(\widehat{K}^2,\widehat{K})=(\widehat{K^{-1}}\widehat{K},\widehat{K})=(\widehat{K}\widehat{K^{-1}},\widehat{K^{-1}})$ by Lemma \ref{Le1}(3). Thus, $$\widehat{K}\widehat{K^{-1}}=|K|\widehat{1}+\alpha\widehat{K}+\alpha\widehat{K^{-1}}+\widehat{S}$$ where $(\widehat{S},\widehat{L})=0$ if $L\in \{1, K, K^{-1}\}$.\\
 
We distinguish between whether $S=\emptyset$ or not. Suppose first that $S=\emptyset$. Since $K^3=KK^2=K(K \cup K^{-1})=\{1\} \cup K \cup K^{-1}$, we obtain by induction that $K^n=\{1\} \cup K \cup K^{-1}$ for all $n\geq 3$. Thus, $\langle K\rangle=KK^{-1}=\{1\} \cup K \cup K^{-1}$. As all non-trivial elements in $\langle K\rangle$ have the same order, it follows that $\langle K\rangle$ is $p$-elementary abelian for some prime $p$, and we have finished. Assume now that $S\neq \emptyset$. We have $$\widehat{K}(\widehat{K}\widehat{K^{-1}})=\widehat{K}(|K|\widehat{1}+\alpha\widehat{K}+\alpha\widehat{K^{-1}}+\widehat{S})=|K|\widehat{K}+\alpha\widehat{K}^2+\alpha\widehat{K}\widehat{K^{-1}}+\widehat{K}\widehat{S}$$ and on the other hand, $$\widehat{K}^2\widehat{K^{-1}}=(\alpha\widehat{K}+\beta\widehat{K^{-1}})\widehat{K^{-1}}=\alpha\widehat{K}\widehat{K^{-1}}+\beta\widehat{K^{-1}}\widehat{K^{-1}}.$$ Taking into account both equalities and that $\widehat{K^{-1}}\widehat{K^{-1}}=\beta\widehat{K}+\alpha\widehat{K^{-1}}$, we obtain $$|K|\widehat{K}+\alpha(\alpha\widehat{K}+\beta\widehat{K^{-1}})+\widehat{K}\widehat{S}=\beta(\beta\widehat{K}+\alpha\widehat{K^{-1}}).$$ If we rearrange, we obtain $$\widehat{K}\widehat{S}=(\beta^2-|K|-\alpha^2)\widehat{K}.$$ In particular, we conclude that $KS=K$ and by applying Lemma \ref{L1}, it easily follows that $\langle S\rangle$ is solvable. Consider now $\overline{G}=G/\langle S\rangle$. We observe from the hypothesis that $\overline{K^2}=\overline{K}\cup \overline{K^{-1}}$, so $\langle \overline{K}\rangle$ is solvable by induction. Consequently, $\langle K\rangle$ is solvable.\\

Now let us see that $x$ is a $p$-element. Since $KS=K$, we have that $K^{-1}K\langle S\rangle=K^{-1}K$ and $KK^{-1}$ is union of left cosets of the subgroup $\langle S\rangle$, so $|\langle S\rangle|$ divides $|KK^{-1}|=1 + 2|K| + |S|$. Hence, $|\langle S\rangle|$ divides $1 + |S|$ because $|\langle S\rangle|$ divides $|K|$. It necessarily follows that $\langle S\rangle=\{1\} \cup S$. On the other hand, since $K^3=KK^2=K(K \cup K^{-1})=\{1\} \cup K \cup K^{-1} \cup S$ and $KS=K$, we easily obtain by induction on $n$ that $K^n=\{1\} \cup K \cup K^{-1} \cup S$ for all $n\geq 3$. Thus, $\langle K\rangle=KK^{-1}=\{1\} \cup K \cup K^{-1} \cup S$. We write $\overline{G}=G/\langle S\rangle$, so $\langle \overline{K}\rangle$ is $p$-elementary abelian for some prime $p$ because $\langle \overline{K}\rangle$ is a minimal normal subgroup with all non-trivial elements of the same order. We write $x=x_px_{p'}$ where $x_p$ and $x_{p'}$ are the $p$-part and the $p'$-part of $x$ respectively. Then $x_{p'}$ and $x_{p'}^{-1}$ are in $\langle S\rangle$. Consequently, $x_p=x x_{p'}^{-1}\in K\langle S\rangle=K$ and so $K$ is a conjugacy class of a $p$-element as required.$\Box$\\

\begin{example} In Theorem D, the case in which $\langle K\rangle$ is non-abelian can happen. We take for instance the group $G=((\mathbb{Z}_2 \times \mathbb{Z}_2 \times \mathbb{Z}_2) \rtimes \mathbb{Z}_7) \rtimes \mathbb{Z}_3=SmallGroup(168,43)$ which has a conjugacy class $K$ of elements of order 7 and size 24 satisfying $K^2=K \cup K^{-1}$. Also, $\langle K\rangle=(\mathbb{Z}_2 \times \mathbb{Z}_2 \times \mathbb{Z}_2) \rtimes \mathbb{Z}_7$.
\end{example}

The following property is useful to check Conjecture \ref{Con3} from the character table, in particular for the sporadic simple groups.

\begin{theorem}\label{Char3} Let $G$ be a finite group and let $K$ be a conjugacy class of an element $x\in G$. Then $K^n=D \cup D^{-1}$ where $D$ is a conjugacy class if and only if there exist positive integers $m_1$ and $m_2$ such that $$\chi(x)^n|K|^n=\chi(1)^{n-1}|D|(m_1\chi(x^n)+m_2\chi(x^{-n}))$$ for all $\chi \in {\rm Irr}(G)$ and $|K|^n=(m_1+m_2)|D|$ where\\

\begin{center} 
\begin{tabular}{cc}
$m_1=\frac{|K|^n}{|G|}\sum_{\chi \in {\rm Irr}(G)}\frac{\chi(x)^n\overline{\chi(x^n)}}{\chi^{n-1}(1)}$ & and $m_2=\frac{|K|^n}{|G|}\sum_{\chi \in {\rm Irr}(G)}\frac{\chi(x)^n\chi(x^{n})}{\chi^{n-1}(1)}.$ \end{tabular}
\end{center}
In particular, $$\chi(x)^n+\chi(x^{-1})^n=\chi(1)^{n-1}(\chi(x^n)+\chi(x^{-n}))$$ for all $\chi \in {\rm Irr}(G)$.
\end{theorem}

\begin{proof} Assume that $K^n=D \cup D^{-1}$ and we write $\widehat{K^n}=m_1\widehat{D}+m_2\widehat{D^{-1}}$ where $m_1$ and $m_2$ are positive integers that can be determined by the character table by using Theorem \ref{multi}. Then $|K|^n=(m_1+m_2)|D|$. Let $\chi \in {\rm Irr}(G)$ and let $R$ and $w_\chi$ be as in Theorem \ref{Char1}. We have \begin{equation*}
R(\widehat{K^n})=R(\widehat{K})^n=m_1R(\widehat{D})+m_2R(\widehat{D^{-1}})
\end{equation*}
and
\begin{equation*}
w_{\chi}(\widehat{K})^n=m_1w_{\chi}(\widehat{D})+m_2w_{\chi}(\widehat{D^{-1}}).
\end{equation*} 
If we suppose that $x^n \in D$ (analogously if $x^n\in D^{-1}$), then $$|K|^{n}\frac{\chi(x)^n}{\chi(1)^n}=m_1\frac{|D|\chi(x^n)}{\chi(1)}+m_2\frac{|D|\chi(x^{-n})}{\chi(1)}$$ and then $$|K|^n\chi(x)^n=\chi(1)^{n-1}|D|(m_1\chi(x^n)+m_2\chi(x^{-n}))$$ for all $\chi \in {\rm Irr}(G)$, as wanted. By taking conjugates in the above equation we obtain $$|K|^n\chi(x^{-1})^n=\chi(1)^{n-1}|D|(m_2\chi(x^n)+m_1\chi(x^{-n}))$$ for all $\chi \in {\rm Irr}(G)$. The last part of the theorem follows by summing the previous equations.\\

Conversely, assume that there exist $m_1$ and $m_2$ satisfying the equalities with characters. We write $K^n=D_1 \cup \cdots \cup D_r$ with $D_i$ a conjugacy class for all $1\leq i\leq r$. We write $\widehat{K^n}=n_1\widehat{D_1}+ \cdots +n_r\widehat{D_r}$ and notice that $|K|^n=n_1|D_1|+\cdots +n_r|D_r|$. Let $\chi \in {\rm Irr}(G)$ and let $R$ and $w_\chi$ be as before. Then \begin{equation*}
R(\widehat{K^n})=R(\widehat{K})^n=n_1R(\widehat{D_1})+\cdots +n_rR(\widehat{D_r})
\end{equation*}
and by hypothesis we know
$$\chi(x)^n|K|^n=\chi(1)^{n-1}|D|(m_1\chi(x^n)+m_2\chi(x^{-n}))$$ for all $\chi \in {\rm Irr}(G)$.
Therefore, 
\begin{equation}\label{CAR1}
|D|m_1\chi(x^n)+|D|m_2\chi(x^{-n})=n_1|D_1|\chi(d_1)+\cdots +n_r|D_r|\chi(d_r).
\end{equation} 
By multiplying both sides by $\overline{\chi(x^n)}$ we get
 \begin{equation*}
|D|m_1\chi(x^n)\overline{\chi(x^n)}+|D|m_2\chi(x^{-n})\overline{\chi(x^n)}=n_1|D_1|\chi(d_1)\overline{\chi(x^n)}+\cdots +n_r|D_r|\chi(d_r)\overline{\chi(x^n)}
\end{equation*} 
From this equation we obtain 

$$|D|m_1\sum_{\chi \in {\rm Irr}(G)}\chi(x^n)\overline{\chi(x^n)}+|D|m_2\sum_{\chi \in {\rm Irr}(G)}\chi(x^{-n})\overline{\chi(x^n)}=|K|^n|{\rm \bf C}_G(x^n)|=$$
$$=n_1|D_1|\sum_{\chi \in {\rm Irr}(G)}\chi(d_1)\overline{\chi(x^n)}+\cdots +n_r|D_r|\sum_{\chi \in {\rm Irr}(G)}\chi(d_r)\overline{\chi(x^n)}.$$
Then $D_i= D=(x^n)^G$ for some $i$. Without loss of generality, we can assume that $D_1=D$. Now, if we multiply both sides of Eq.(\ref{CAR1}) by $\chi(x^n)$ and argue similarly we conclude that $D_2=D^{-1}$ and $n_i=0$ for all $i\neq 1,2$. This means that $K^n=D \cup D^{-1}$.
\end{proof}

\begin{remark} Let $G$ be a group and let $K$ be a conjugacy class of an element $x\in G$. If $K^n=D \cup D^{-1}$ for some $n\in \mathbb{N}$, $n\geq 2$ and $D$ a conjugacy class, then $G$ is not a sporadic simple group. 
\end{remark}

\begin{proof}
Let $x, x^n\in G$ such that $K=x^G$, $D=(x^n)^G$ and $K^n=D \cup D^{-1}$ for some $n \in \mathbb{N}$. We show that for any sporadic simple group there is no conjugacy class satisfying the hypotheses of the theorem. By Theorem \ref{Char3}, we know that the hypotheses imply
\begin{equation}\label{EC1}
\chi(x)^n+\chi(x^{-1})^n=\chi(1)^{n-1}(\chi(x^n)+\chi(x^{-n}))
\end{equation} 
for all $\chi \in {\rm Irr}(G)$. The aim is to find some irreducible character that does not satisfy Eq.(\ref{EC1}). Recall that the smallest integer $m$ satisfying $C^m=G$ for each non-trivial conjugacy class $C$ of $G$ is called the covering number of $G$. The covering number of each sporadic simple group is at most 6 (\cite{Zisser} and \cite{products}). It can be checked by using the character tables (for example included in GAP) that for any of these groups and any two non-trivial conjugacy classes of it and $n<6$, there exists an irreducible character which does not satisfy Eq.(\ref{EC1}).
\end{proof}


\begin{thebibliography}{99}


\bibitem{AradFisman} Z. Arad and E. Fisman, An analogy between products of two conjugacy classes and products of two irreducible characters in finite groups. \textit{Proc. Edinb. Math. Soc.} \textbf{30} 7-22 (1987).

\bibitem{products} Z. Arad and M. Herzog, \textit{Products of conjugacy classes in groups}, Lecture Notes in Mathematics, 1112, Springer-Verlag, Berlin, (1985).


\bibitem{Nuestro5} A. Beltr\'an, M.J. Felipe and C. Melchor, Squares of real conjugacy classes in finite groups, \textit{Ann. Mat. Pura Appl.} \textbf{197} (2) 317-328 (2018).

\bibitem{Nuestro6} A. Beltr\'an, M.J. Felipe and C. Melchor, Multiplying a conjugacy class by its inverse in a finite group,  \textit{Israel J. Math.} \textbf{227}(2) 811-825 (2018).


\bibitem{GAP} The GAP Group, GAP - Groups, Algorithms and
Programming, Vers. 4.7.7; 2015. (http://www.gap-system.org)

\bibitem{Goreinstein} D. Gorenstein, \textit{Finite groups}, AMS Chelsea Publishing, (1980).

\bibitem{GurMalle} R.M. Guralnick and G. Malle, Variations on the Baer-Suzuki theorem, {\it Math. Z.} \textbf{279} 981-1006 (2015).

\bibitem{GurRob} R.M. Guralnick and G.R. Robinson, On extensions on the Baer-Suzuki theorem, {\it Israel J. Math.} \textbf{82} 281-297 (1993).

\bibitem{GuralnickNavarro} R.M. Guralnick and G. Navarro, Squaring a conjugacy class and cosets of normal subgroups, Proc. \textit{Am. Math. Soc.}  {\bf 144} (5) 1939-1945 (2016).

\bibitem{Huppert} B. Huppert, {\em Character Theory of Finite Groups}, Walter de Gruyter, Berl\'in. New York, (1998).

\bibitem{Isaacs} I.M. Isaacs, \textit{Character theory of finite groups}, Academic Press, Inc, New York, (1976).

\bibitem{MooriViet} J. Moori, H.P. Tong-Viet, Products of conjugacy classes in simple groups, {\it Quaest. Math.} \textbf{34} (4) 433-439 (2011).

\bibitem{Zisser} I. Zisser, The covering numbers of the sporadic simple groups, \textit{Israel J. Math.}  {\bf 67} (2) 217-224 (1989).
\end{thebibliography}

%
%

\end{document}